\documentclass[12pt]{amsart}
\usepackage{amscd,amsmath,amsthm,amssymb,graphics}
\usepackage{pstcol,pst-plot,pst-3d}
\usepackage{lmodern,pst-node}
\usepackage{multicol}
\usepackage{epic,eepic}
\usepackage{amsfonts,amssymb,amscd,amsmath,enumerate,verbatim}
\psset{unit=0.7cm,linewidth=0.8pt,arrowsize=2.5pt 4}

\newpsstyle{fatline}{linewidth=1.5pt}
\newpsstyle{fyp}{fillstyle=solid,fillcolor=verylight}
\definecolor{verylight}{gray}{0.97}
\definecolor{light}{gray}{0.9}
\definecolor{medium}{gray}{0.85}
\definecolor{dark}{gray}{0.6}

\def\NZQ{\Bbb}               
\def\NN{{\NZQ N}}

\def\ZZ{{\NZQ Z}}

\def\frk{\frak}               

\def\Phi{{\frk n}}
\def\Phi{{\frk N}}
\def\MI{{\mathcal I}}

\def\MP{{\mathcal P}}
\def\MQ{{\mathcal Q}}

\def\MH{{\mathcal H}}
\def\MS{{\mathcal S}}
\def\ML{{\mathcal L}}
\def\MC{{\mathcal C}}
\def\MM{{\mathcal M}}
\def\MB{{\mathcal B}}

\def\opn#1#2{\def#1{\operatorname{#2}}} 
\opn\chara{char} \opn\length{\ell} \opn\pd{pd} \opn\rk{rk}
\opn\projdim{proj\,dim} \opn\injdim{inj\,dim} \opn\rank{rank}
\opn\depth{depth} \opn\grade{grade} \opn\height{height}
\opn\embdim{emb\,dim} \opn\codim{codim}

\opn\Tr{Tr} \opn\bigrank{big\,rank}
\opn\superheight{superheight}\opn\lcm{lcm}
\opn\trdeg{tr\,deg}
\opn\reg{reg} \opn\lreg{lreg} \opn\ini{in} \opn\lpd{lpd}
\opn\size{size}\opn\bigsize{bigsize}
\opn\cosize{cosize}\opn\bigcosize{bigcosize}
\opn\sdepth{sdepth}\opn\sreg{sreg}
\opn\link{link}\opn\fdepth{fdepth}
\opn\div{div} \opn\Div{Div} \opn\cl{cl} \opn\Cl{Cl}
\opn\Spec{Spec} \opn\Supp{Supp} \opn\supp{supp} \opn\Sing{Sing}
\opn\Ass{Ass} \opn\Min{Min}\opn\Mon{Mon} \opn\dstab{dstab} \opn\astab{astab}
\opn\Syz{Syz}
\opn\Ann{Ann} \opn\Rad{Rad} \opn\Soc{Soc}
\opn\Im{Im} \opn\Ker{Ker} \opn\Coker{Coker} \opn\Am{Am}
\opn\Hom{Hom} \opn\Tor{Tor} \opn\Ext{Ext} \opn\End{End}
\opn\Aut{Aut} \opn\id{id}

\opn\nat{nat}
\opn\pff{pf}
\opn\Pf{Pf} \opn\GL{GL} \opn\SL{SL} \opn\mod{mod} \opn\ord{ord}
\opn\Gin{Gin} \opn\Hilb{Hilb}\opn\sort{sort}
\opn\initial{init}
\opn\ende{end}
\opn\height{height}
\opn\aff{aff} \opn\con{conv} \opn\relint{relint} \opn\st{st}
\opn\lk{lk} \opn\cn{cn} \opn\core{core} \opn\vol{vol}
\opn\link{link} \opn\star{star}\opn\lex{lex}
\opn\sign{sign}
\opn\gr{gr}

\def\pot#1#2{#1[\kern-0.28ex[#2]\kern-0.28ex]}

\opn\dirlim{\underrightarrow{\lim}}
\opn\inivlim{\underleftarrow{\lim}}
\let\sect=\cap

\let\Union=\bigcup

\let\to=\rightarrow

\def\Implies{\ifmmode\Longrightarrow \else
        \unskip${}\Longrightarrow{}$\ignorespaces\fi}
\def\implies{\ifmmode\Rightarrow \else
        \unskip${}\Rightarrow{}$\ignorespaces\fi}
\def\iff{\ifmmode\Longleftrightarrow \else
        \unskip${}\Longleftrightarrow{}$\ignorespaces\fi}

\let\:=\colon
\newtheorem{Theorem}{Theorem}[section]
 \newtheorem{Lemma}[Theorem]{Lemma}
 \newtheorem{Corollary}[Theorem]{Corollary}
 \newtheorem{Proposition}[Theorem]{Proposition}

\let\epsilon\varepsilon
\let\kappa=\varkappa
\def\qed{\ifhmode\textqed\fi
      \ifmmode\ifinner\quad\qedsymbol\else\dispqed\fi\fi}
\def\textqed{\unskip\nobreak\penalty50
       \hskip2em\hbox{}\nobreak\hfil\qedsymbol
       \parfillskip=0pt \finalhyphendemerits=0}
\def\dispqed{\rlap{\qquad\qedsymbol}}

\opn\dis{dis}
\def\pnt{{\raise0.5mm\hbox{\large\bf.}}}

\opn\Lex{Lex}

\begin{document}
 \title{The coordinate ring of a simple polyomino}

 \author {J\"urgen Herzog and Sara Saeedi Madani}

\address{J\"urgen Herzog, Fachbereich Mathematik, Universit\"at Duisburg-Essen, Campus Essen, 45117
Essen, Germany} \email{juergen.herzog@uni-essen.de}

\address{Sara Saeedi Madani, School of Mathematics,
Institute for Research in Fundamental Sciences (IPM), P.O. Box 19395-5746, Tehran, Iran} \email{sarasaeedim@gmail.com}

\thanks{}
\thanks{The paper was written while the second author was visiting the Department of Mathematics of University Duisburg-Essen. She wants to express her thanks for its hospitality.}

\begin{abstract}
In this paper it is shown that a polyomino is balanced if and only if it is simple. As a consequence one obtains that the coordinate ring of a
simple polyomino is a normal Cohen--Macaulay domain.
\end{abstract}

\thanks{}


 \maketitle

\section*{Introduction}
\label{introduction}
The study of the algebraic properties of ideals of $t$-minors  of an $(m\times n)$-matrix of indeterminates is a classical subject of research  in Commutative Algebra. The basic reference on this subject is \cite{BV}. Gr\"obner bases of determinantal ideals are treated in  \cite{HT} and ladder determinantal ideals are considered in  \cite{C}. In these  articles the reader finds further  references to other aspects of determinantal ideals.  Hochster and Eagon \cite{HE} showed that determinantal ideals define normal Cohen--Macaulay domains.  There are various generalizations of this result which include a similar statement as that of Hochster and Eagon for ideals of minors of ladders. Ladders may be viewed as special classes of  polyominoes, which roughly speaking are figures obtained by joining squares of equal size edge to edge. The squares which establish a polyomino are called its cells. The precise definitions are given in Section~\ref{preliminaries}. Polyminoes which originally were considered  in recreational mathematics have been and still are subject of intense research in connection with tiling problems of the plane, see for example \cite{Go} and \cite{GS}.

Let $\MP$ be a polyomino. We fix a field $K$ and consider  in a suitable polynomial ring $S$ over $K$ the ideal of all $t$-minors belonging to $\MP$. It is natural to ask for which shape of the polyomino this ideal of $t$-minors defines a Cohen--Macaulay domain as it is the case for a matrix or a ladder.
Here we restrict our attention to the ideal of all $2$-minors of a polyomino. The $2$-minors belonging to a polymino $\MP$, are called the inner minors, and the ideal $I_\MP$ they generate is called the ideal of inner minors of $\MP$ or the  polyomino ideal attached to $\MP$. The residue class ring $K[\MP]$ defined by  the polyomino ideal is called the coordinate ring of $\MP$.

Polyomino ideals attached to polyominoes have been introduced by Qureshi in \cite{Q} where,  among other results, she showed that the coordinate ring of a convex polyomino is a normal Cohen--Macaulay domain and where for stack polyominoes she computed the divisor class group and determined those stack polyminoes which are Gorenstein. A classification of convex polyominoes whose polyomino ideal is linearly related is given in \cite{EHH}. In a subsequent paper  \cite{HQSh} of Qureshi with Shikama and the first author of this paper,  balanced polyominoes were introduced.  To define a balanced polyomino, one labels the  vertices of a polyomino by integer numbers in a way that row and column sums are zero along intervals that belong to the polyomino. Such a labeling is called admissible. To each admissible labeling $\alpha$, a binomial $f_\alpha$ is naturally associated. The ideal $J_\MP$ generated  by the $f_\alpha$ generates the lattice ideal of a certain saturated lattice $\Lambda\subset \ZZ^q$ for a suitable $q$. Balanced polyominoes are exactly those for which $I_\MP=J_\MP$. Since the lattice ideal of a saturated lattice is always a prime ideal it follows that $K[\MP]$ is a domain if $\MP$ is balanced. Actually in \cite{HQSh} it is even shown that  $K[\MP]$ is a normal Cohen--Macaulay domain if $\MP$ is balanced.

In \cite{HQSh} it is conjectured that a polyomino is balanced if and only if it is simple. A polyomino is called simple if it is hole-free. The main result of this paper is Theorem~\ref{main} in which we prove the above conjecture. As a consequence we obtain that the coordinate ring of a simple polyomino is a normal Cohen--Macaulay domain. This result covers the case of row or column convex polyominoes as well as of tree--like polyominoes which are treated in \cite{HQSh}. We also would like to mention that there are some examples of polyominoes with holes whose coordinate rings nevertheless are not domains. Thus it remains an open problem to classify all polyominoes whose coordinate rings are domains.

The proof of our main result requires some combinatorial geometric arguments which by a lack of suitable references we included to this paper. The first fact needed is that the border of a simple polyomino is a simple rectilinear polygon, in other words, a polygon which does not self-intersect and whose edges intersect orthogonally.   This fact allows us to define an admissible border labeling which is crucial in the proof of the main theorem. We call a corner $c$ of a  rectilinear polygon ``good" if the rectangle spanned by $c$ and its neighbor corners belongs to the interior of the polygon. The other fact needed in the proof is that any rectilinear polygon has at least four good corners. In Computational Geometry, the rectilinear polygons are studied in connection to the so called art gallery problem. They are also used in computer aided manufacturing processes.

\section{Preliminaries on polyominoes, rectilinear polygons and related algebraic concepts}
\label{preliminaries}

In this section we introduce simple and balanced polyominoes and present some of their properties and related facts which are needed in the next section.

Let $\mathbb{R}_{+}^2=\{(x,y)\in \mathbb{R}^2:x,y\geq 0\}$. We consider $(\mathbb{R}_{+}^2,\leq )$ as a partially ordered set with $(x,y)\leq (z,w)$ if $x\leq z$ and $y\leq w$. Let $a,b\in \mathbb{N}^2$ (by $\NN$, we mean the set of all nonnegative integers). Then the set $[a,b]=\{c\in \mathbb{N}^2: a\leq c\leq b\}$ is called an \textit{interval}. In what follows it is convenient also to define $[a,b]$ to be $[b,a]$ if $b\leq a$. Furthermore, we set $\overline{[a,b]}=\{x\in \mathbb{R}^2:a\leq x\leq b\}$.

Let $a=(i,j),b=(k,l)\in \mathbb{N}^2$ with $i<k$ and $j<l$. Then the elements $a$ and $b$ are called \textit{diagonal corners}, and the elements $c=(i,l)$ and $d=(k,j)$ are called \textit{anti-diagonal corners} of $[a,b]$.

A \textit{cell} $C$ is an interval of the form $[a,b]$, where $b=a+(1,1)$. The elements of $C$ are called \textit{vertices} of $C$. We denote the set of vertices of $C$ by $V(C)$. The intervals $[a,a+(1,0)]$, $[a+(1,0),a+(1,1)]$, $[a+(0,1),a+(1,1)]$ and $[a,a+(0,1)]$ are called \textit{edges} of $C$.

Let $\MP$ be a finite collection of cells of $\mathbb{N}^2$. Then two cells $C$ and $D$ are called \textit{connected} if there exists a sequence
$\MC:C=C_1,C_2,\ldots,C_t=D$ of cells of $\MP$ such that for all $i=1,\ldots,t-1$ the cells $C_i$ and $C_{i+1}$ intersect in an edge. If the cells in $\MC$ are pairwise distinct, then $\MC$ is called a \textit{path} between $C$ and $D$. A finite collection of cells $\MP$ is called a \textit{polyomino} if every two cells of $\MP$ are connected. The \textit{vertex set} of $\MP$, denoted $V(\MP)$, is defined to be $\bigcup_{C\in \MP}V(C)$. The area $\overline{\MP}$ covered by $\MP$ is given by $\bigcup_{C\in \MP}\overline{C}$. Figure~\ref{polyomino} shows a polyomino whose cells are marked by gray color.

A \textit{rectangular polyomino} is defined to be the collection of all cells inside an interval.

Let  $\MQ$ be an arbitrary collection of cells. Then each connected component of $\MQ$ is a polyomino.

\begin{figure}[hbt]
\begin{center}
\psset{unit=0.8cm}
\begin{pspicture}(2.5,-1)(2.5,5)
\psline(1,0)(3,0)
\psline(0,1)(4,1)
\psline(0,2)(5,2)
\psline(0,3)(5,3)
\psline(1,4)(4,4)
\psline(0,1)(0,3)
\psline(1,0)(1,4)
\psline(2,0)(2,4)
\psline(3,0)(3,4)
\psline(4,1)(4,4)
\psline(5,2)(5,3)
\pspolygon[style=fyp,fillcolor=light](1,0)(1,1)(2,1)(2,0)
\pspolygon[style=fyp,fillcolor=light](2,0)(2,1)(3,1)(3,0)
\pspolygon[style=fyp,fillcolor=light](0,1)(0,2)(1,2)(1,1)
\pspolygon[style=fyp,fillcolor=light](1,1)(1,2)(2,2)(2,1)
\pspolygon[style=fyp,fillcolor=light](3,1)(3,2)(4,2)(4,1)
\pspolygon[style=fyp,fillcolor=light](0,2)(0,3)(1,3)(1,2)
\pspolygon[style=fyp,fillcolor=light](1,2)(1,3)(2,3)(2,2)
\pspolygon[style=fyp,fillcolor=light](2,2)(2,3)(3,3)(3,2)
\pspolygon[style=fyp,fillcolor=light](3,2)(3,3)(4,3)(4,2)
\pspolygon[style=fyp,fillcolor=light](4,2)(4,3)(5,3)(5,2)
\pspolygon[style=fyp,fillcolor=light](1,3)(1,4)(2,4)(2,3)
\pspolygon[style=fyp,fillcolor=light](2,3)(2,4)(3,4)(3,3)
\pspolygon[style=fyp,fillcolor=light](3,3)(3,4)(4,4)(4,3)
\end{pspicture}
\end{center}
\caption{A polyomino}\label{polyomino}
\end{figure}
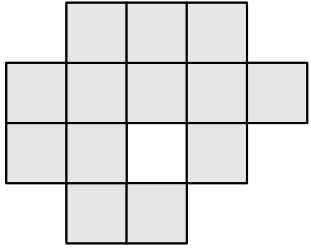

An interval $[a,b]$ with $a=(i,j)$ and $b=(k,j)$ is called a \textit{horizontal edge interval} of $\MP$ if the intervals $[(t,j),(t+1,j)]$ for $t=i,\ldots,k-1$ are edges of cells of $\MP$. Similarly, a \textit{vertical edge interval} of $\MP$ is defined to be an interval $[a,b]$ with $a=(i,j)$ and $b=(i,l)$ such that the intervals $[(i,t),(i,t+1)]$ for $t=j,\ldots,l-1$ are edges of cells of $\MP$.

We call an edge of a cell $C$ of $\MP$ a \textit{border edge} if it is not an edge of any other cell, and define the \textit{border} of $\MP$ to be the union of all border edges of $\MP$. A \textit{horizontal border edge interval} of $\MP$ is defined to be a horizontal edge interval of $\MP$ whose edges are border edges. Similarly, we define a \textit{vertical border edge interval} of $\MP$.

\medskip
Let $\MP$ be a polyomino and $\MI$ a rectangular polyomino such that $\MP\subset \MI$. Then the polyomino $\MP$  is called \textit{simple}, if each cell $C$ which does not belong to $\MP$ satisfies the following condition $(*)$: there exists a path $\MC: C=C_1,C_2,\ldots,C_t=D$ with $C_i\not \in \MP$ for all $i=1,\ldots,t$ and such that $D$ is not a cell of $\MI$. For example, the polyomino which is shown in Figure~\ref{polyomino} is not simple, while Figure~\ref{simple} shows a simple polyomino.

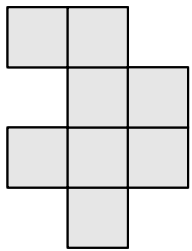
\begin{figure}[hbt]
\begin{center}
\psset{unit=0.8cm}
\begin{pspicture}(2,-1)(2,5)
\psline(1,0)(2,0)
\psline(0,1)(3,1)
\psline(0,2)(3,2)
\psline(0,3)(3,3)
\psline(0,4)(2,4)
\psline(0,1)(0,2)
\psline(0,3)(0,4)
\psline(1,0)(1,4)
\psline(2,0)(2,4)
\psline(3,1)(3,3)
\pspolygon[style=fyp,fillcolor=light](1,0)(1,1)(2,1)(2,0)
\pspolygon[style=fyp,fillcolor=light](0,1)(0,2)(1,2)(1,1)
\pspolygon[style=fyp,fillcolor=light](1,1)(1,2)(2,2)(2,1)
\pspolygon[style=fyp,fillcolor=light](2,1)(2,2)(3,2)(3,1)
\pspolygon[style=fyp,fillcolor=light](1,2)(1,3)(2,3)(2,2)
\pspolygon[style=fyp,fillcolor=light](2,2)(2,3)(3,3)(3,2)
\pspolygon[style=fyp,fillcolor=light](0,3)(0,4)(1,4)(1,3)
\pspolygon[style=fyp,fillcolor=light](1,3)(1,4)(2,4)(2,3)
\end{pspicture}
\end{center}
\caption{A simple polyomino}\label{simple}
\end{figure}

Let $\MP$ be a polyomino and let $\MH$ be the collection of cells $C\notin \MP$ which do not satisfy condition $(*)$. The connected components of $\MH$ are called the \textit{holes} of $\MP$. For example, the polyomino which is shown in Figure~\ref{polyomino} has exactly one hole consisting of just one cell. Note that $\MP$ is simple if and only if it is hole-free. Each hole of $\MP$ is a polyomino. In fact, even one has

\begin{Lemma}
\label{hole}
Each hole of a polyomino is a simple polyomino.
\end{Lemma}

\begin{proof}
Let $\MP'$ be a hole of the simple polyomino $\MP$, and assume that $\MP'$ is not simple. Let $\MP''$ be a hole of $\MP'$. Then $\MP''$ is again a polyomino. Let $C$ be a cell of $\MP''$ which has a border edge of $\MP''$. Then $C$ shares an edge with a cell $D\in\MP'$. Since $\MP'$ is a connected component of the set $\MH$ of cells not belonging to $\MP$ which do not satisfy condition $(*)$ and since $C$ has a common edge with $D$ it follows that $C\in \MP$. However since $\MP$ is connected there exists a path of cells which all belong to $\MP$ and which connect $C$ with a cell of $\MP\setminus \MP''$, contradicting the fact that $\MP''$ is a hole.
\end{proof}

The polyomino in Figure \ref{polyomino} has two cells intersecting in only one vertex which does not belong to any other cell. This can not happen if the polyomino is simple.

\begin{Lemma}
\label{meet}
Let $\MP$ be a simple polyomino. Then there does not exist any vertex $v$ which belongs to exactly two cells $C$ and $C'$ of $\MP$ such that $C\sect C'=\{v\}$.
\end{Lemma}

\begin{proof}
Suppose on the contrary that there exists such a vertex $v$. According to Figure~\ref{v}, the only cells of $\MP$ which contain $v$ could be the four cells $C$, $C'$, $D$ and $D'$. By our assumption, we may assume that $C$ and $C'$ belong to $\MP$ and $D$ and $D'$ do not belong to $\MP$. Since $\MP$ is a polyomino, there exists a path of cells of $\MP$ connecting $C$ and $C'$. Thus, either $D$ or $D'$ is contained in a hole of $\MP$. It contradicts the fact that $\MP$ is a simple polyomino.
\begin{figure}[hbt]
\begin{center}
\psset{unit=0.9cm}
\begin{pspicture}(-1,-2)(2,1)
\psline(-1,0)(0,0)
\psline(-1,0)(-1,-1)
\psline(-1,-1)(0,-1)
\psline(0,-1)(0,0)
\psline(0,0)(0,1)
\psline(0,1)(1,1)
\psline(0,0)(1,0)
\psline(1,0)(1,1)
\psline[linestyle=dashed](-1,1)(0,1)
\psline[linestyle=dashed](-1,1)(-1,0)
\psline[linestyle=dashed](1,-1)(1,0)
\psline[linestyle=dashed](0,-1)(1,-1)
\rput(0.5,0.5){$C'$}
\rput(-0.5,-0.5){$C$}
\rput(-0.5,0.5){$D$}
\rput(0.5,-0.5){$D'$}
\rput(0.2,-0.2){$v$}
\rput(0,0){$\bullet$}
\end{pspicture}
\end{center}
\caption{Two cells $C$ and $C'$ belong to $\MP$}\label{v}
\end{figure}
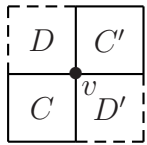
\end{proof}

\begin{Corollary}
\label{endpoint}
Let $\MP$ be a simple polyomino and let $I$ and $I'$ be  two distinct maximal border edge intervals of $\MP$ with $I\sect I'\neq \emptyset$. Then their intersection is a common endpoint of $I$ and $I'$. Furthermore, at most two maximal border edge intervals of $\MP$ have a nontrivial  intersection.
\end{Corollary}

\begin{proof}
Let $I=[a,b]$ and $I'=[c,d]$. The edge intervals $I$ and $I'$ are not both horizontal or vertical edge intervals, since otherwise their maximality implies that they are disjoint. Suppose that $I$ is a horizontal edge interval and $I'$ is a vertical edge interval. So, obviously, they intersect in one vertex, say $v$. Suppose that $v$ is not an endpoint of $I$ or $I'$. If $v$ is an endpoint of just one of them, then without loss of generality, we may assume that we are in the case which is shown on the left hand side of Figure~\ref{intervals}. Thus, since $I$ and $I'$ are maximal border edge intervals, it follows that among the four possible cells of $\NN^2$ which contain $v$, exactly one of them belongs to $\MP$, which is a contradiction. If $v$ is not an endpoint of any of $I$ and $I'$, then we are in the case which is displayed on the right hand side of Figure~\ref{intervals}. Among four possible cells of $\NN^2$ which contain $v$, only a pair of them, say $C$ and $C'$, with $C\sect C'=\{v\}$, belong to $\MP$, since the edges of $I$ and $I'$ are all border edges. But, by Lemma~\ref{meet}, this is impossible, since $\MP$ is simple. Thus, $v$ has to be  a common endpoint of $I$ and $I'$.
\begin{figure}[hbt]
\begin{center}
\psset{unit=0.9cm}
\begin{pspicture}(-0.2,-2)(2,3)
\psline(-1,0)(2,0)
\psline(0,0)(0,2)
\rput(-1.25,0){$a$}
\rput(2.25,0){$b$}
\rput(0,2.3){$d$}
\rput(0,-0.25){$v=c$}
\rput(0,0){$\bullet$}
\rput(-1,0){$\bullet$}
\rput(2,0){$\bullet$}
\rput(0,2){$\bullet$}
\end{pspicture}
\psset{unit=0.9cm}
\begin{pspicture}(-5.5,-2)(2,3)
\psline(-1,0)(2,0)
\psline(0,2)(0,-1)
\rput(-1.25,0){$a$}
\rput(2.25,0){$b$}
\rput(0,2.3){$d$}
\rput(0.25,-0.25){$v$}
\rput(0,-01.25){$c$}
\rput(0,0){$\bullet$}
\rput(-1,0){$\bullet$}
\rput(2,0){$\bullet$}
\rput(0,2){$\bullet$}
\rput(0,-1){$\bullet$}
\end{pspicture}
\end{center}
\caption{The vertex $v$ is not a common endpoint of $I$ and $I'$}\label{intervals}
\end{figure}
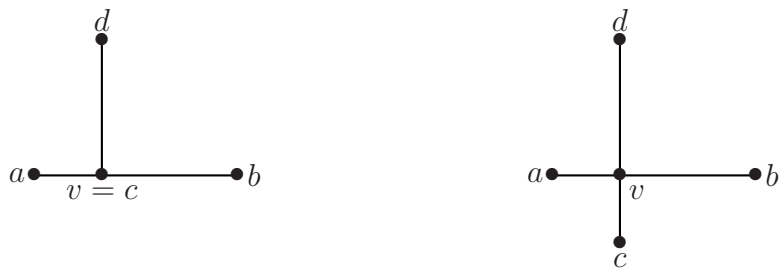

Now, suppose more than two maximal border edge intervals have a nontrivial intersection. Then this intersection is a common endpoint of these intervals. Thus at least two of these intervals are either horizontal or vertical, contradicting the fact that they are all maximal.
\end{proof}

Now, we present some concepts and facts about rectilinear polygons which are used in the course of the proof of the main result of this paper.

A \textit{rectilinear polygon} is a polygon whose edges meet orthogonally. It is easily seen that the number of edges of a rectilinear polygon is even. Note that rectilinear polygons are also known as \textit{orthogonal polygons}. A rectilinear polygon is shown in Figure~\ref{rectilinear}.

\begin{figure}[hbt]
\begin{center}
\psset{unit=0.4cm}
\begin{pspicture}(3.5,0)(4.5,10)
\rput(1.25,0)
{
\pspolygon(-2,3)(3,3)(3,4)(4,4)(4,1)(2,1)(2,0)(6,0)(6,5)(8,5)(8,7)(4,7)(4,10)(1,10)(1,8)(-2,8)(-2,3)
}
\end{pspicture}
\end{center}
\caption{A rectilinear polygon}\label{rectilinear}
\end{figure}
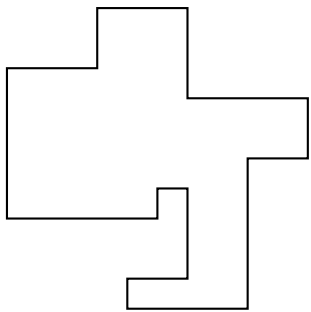

A rectilinear polygon is called \textit{simple} if it does not self-intersect. The rectilinear polygon in Figure~\ref{rectilinear} is a simple rectilinear polygon.

Let $R$ be a  simple rectilinear polygon. The bounded area whose border is $R$ is called the \textit{interior} of $R$. By the \textit{open interior} of $R$ we mean the interior of $R$ without its boundary.

A simple rectilinear polygon has two types of corners: the corners in which the smaller angle ($90$ degrees) is interior to the polygon are called \textit{convex corners}, and the corners in which the larger angle ($270$ degrees) is interior to the polygon are called \textit{concave corners}.

\medskip
Let $E_1,\ldots,E_m$ be the border edges of $\MP$. Then we set $B(\MP)=\bigcup_{i=1}^m\overline{E}_i$. Observe that the border of $\MP$ as defined before  is the set of lattice points which belong to $B(\MP)$.

\begin{Lemma}
\label{borderpolygon}
Let $\MP$ be a simple polyomino. Then $B(\MP)$ is a simple rectilinear polygon.
\end{Lemma}

\begin{proof}
First we show that for each maximal horizontal (resp. vertical) border edge interval $I=[a,b]$ of $\MP$, there exists a unique maximal vertical (resp. horizontal) border edge interval $I'$ such that $a$ is an endpoint of it. By Corollary~\ref{endpoint} the vertex $a$  is then the  endpoint of precisely $I$ and $I'$.   Without loss of generality let $I=[a,b]$ be a horizontal maximal border edge interval of $\MP$. Let $C$ be the only cell of $\MP$ for which $a$ is a vertex, and which has a border edge contained in $I$. First we assume that $a$ is a diagonal corner of $C$ which implies that $C$ is upside of $I$, see Figure~\ref{lemma1}. The argument of the other case in which $a$ is an anti-diagonal corner of $C$, and hence $C$ is downside of $I$, is similar.
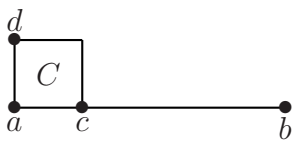
\begin{figure}[hbt]
\begin{center}
\psset{unit=0.9cm}
\begin{pspicture}(2,0)(2,2)
\psline(0,0)(4,0)
\psline(0,0)(0,1)
\psline(0,1)(1,1)
\psline(1,0)(1,1)
\rput(0,-0.25){$a$}
\rput(4,-0.3){$b$}
\rput(1,-0.25){$c$}
\rput(0,1.3){$d$}
\rput(0.5,0.5){$C$}
\rput(0,0){$\bullet$}
\rput(1,0){$\bullet$}
\rput(4,0){$\bullet$}
\rput(0,1){$\bullet$}
\end{pspicture}
\end{center}
\caption{The interval $[a,b]$ and a cell $C$}\label{lemma1}
\end{figure}

Referring to Figure~\ref{lemma1}, we distinguish two cases: either the unique cell $D$, different from $C$ sharing the edge $[a,d]$ with $C$, belongs to $\MP$ or not.

Let us first assume that $D\notin \MP$. Then $[a,d]$ is a border edge of $\MP$, and hence it is contained in a maximal vertical border edge interval $I'$ of $\MP$ such that by Corollary~\ref{endpoint}, $a$ is an endpoint of $I'$. Hence $I'$ is the unique maximal vertical border edge interval of $\MP$ for which $a$ is an endpoint.

Next assume that $D\in \MP$. Then the cell $C'$ belongs to $\MP$ (see Figure~\ref{lemma3}), because  $[a,b]$ is a maximal horizontal border edge interval, so that   $[e,a]$ can not be a border edge. The edge $[f,a]$ is a border edge, since otherwise there is a cell containing both of the edges $[f,a]$ and $[a,c]$, contradicting the fact that $[a,c]$ is a border edge. Therefore, there exists the unique maximal vertical border edge interval $I'$ which contains $[f,a]$ such that $a$ is an endpoint of $I'$.
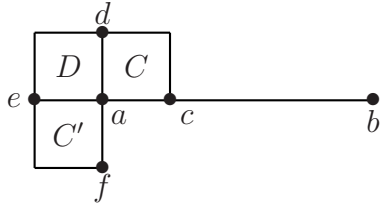
\begin{figure}[hbt]
\begin{center}
\psset{unit=0.9cm}
\begin{pspicture}(1.5,-2)(1.5,2)
\psline(0,0)(4,0)
\psline(-1,0)(0,0)
\psline(-1,0)(-1,-1)
\psline(-1,-1)(0,-1)
\psline(0,-1)(0,0)
\psline(0,0)(0,1)
\psline(0,1)(1,1)
\psline(1,0)(1,1)
\psline(-1,1)(0,1)
\psline(-1,1)(-1,0)
\rput(0.5,0.5){$C$}
\rput(-0.5,-0.5){$C'$}
\rput(-0.5,0.5){$D$}
\rput(0.25,-0.25){$a$}
\rput(4,-0.3){$b$}
\rput(1.25,-0.25){$c$}
\rput(0,1.3){$d$}
\rput(0,-1.3){$f$}
\rput(-1.3,0){$e$}
\rput(0,0){$\bullet$}
\rput(1,0){$\bullet$}
\rput(4,0){$\bullet$}
\rput(0,1){$\bullet$}
\rput(-1,0){$\bullet$}
\rput(0,-1){$\bullet$}
\end{pspicture}
\end{center}
\caption{Intervals $[a,c]$ and $[f,a]$ are two border edges}\label{lemma3}
\end{figure}

The same argument can be applied for $b$ to show that $b$ is also just the endpoint of $I$ and of a unique maximal vertical border edge interval $I'$ of $\MP$.

Now, let $I_1$ be a maximal horizontal border edge interval of $\MP$. By what we have shown before, there exists a unique sequence of maximal border edge intervals $I_1,I_2,\ldots$ of $\MP$ with $I_i=[a_i,a_{i+1}]$ such that they are alternatively horizontal and vertical. Since $V(\MP)$ is finite, there exists a smallest integer $r$ such that for some $i<r-1$, $I_i\cap I_r\neq \emptyset$. Since $I_i$ and $I_r$ are distinct maximal border edge intervals of $\MP$, they intersect in one of their endpoints, by Corollary~\ref{endpoint}. Thus, $I_i\cap I_r=\{a_i\}$, since $r\neq i$ and by Corollary~\ref{endpoint}, $a_{i+1}$ can not be a common vertex between three maximal border edge intervals $I_i$, $I_{i+1}$ and $I_r$. It follows that $i=1$, since otherwise $a_i$ also belong to $I_{i-1}$ which is a contradiction, by Corollary~\ref{endpoint}.

Our discussion shows that $R=\Union_{j=1}^r\bar{I_j}$ is a simple rectilinear polygon. Suppose that  $R\neq B(\MP)$. Then there exists a maximal border edge interval $I_1'$ which is different from the intervals $I_j$. As we did for $I_1$ we may start with $I_1'$ to construct a sequence of border edge intervals $I'_j$ to obtain a simple rectilinear  polygon $R'$ whose edges are formed by some maximal  border edge intervals of $\MP$. We claim that $R\sect R'=\emptyset$. Suppose this is not the case, then $I_j\sect I'_k\neq \emptyset$  for some $j$ and $k$, and hence  by Corollary~\ref{endpoint} these two intervals meet at a common endpoint. Thus it follows that $I'_k$ also has a common intersection with one of the neighbor intervals $I_t$ of $I_j$, contradicting the fact that no three maximal border edge intervals intersect nontrivially, see Corollary~\ref{endpoint}. Hence $R\sect R'=\emptyset$, as we claimed.

All the cells of the interior of $R$ must belong to $\MP$, because otherwise $\MP$ is not simple. It follows that $R'$ does not belong to the interior of $R$, and vice versa. Thus the interior cells  of $R$ and $R'$ form two disjoint sets of cells of $\MP$. Since $\MP$ is a polyomino, there exists a path of cells connecting the interior cells of $R$ with those of $R'$. The edges where this path meets $R$ and $R'$ can not be border edges, a contradiction. Thus we conclude that $R=B(\MP)$.
\end{proof}

For a polyomino $\MP$, a function $\alpha\: V(\MP) \to \ZZ$ is called an \textit{admissible labeling} of $\MP$ (see \cite{Q}), if for all maximal horizontal and vertical edge intervals $I$ of $\MP$, we have
\[
\sum_{a\in I}\alpha(a)=0.
\]
In Figure~\ref{admissible} an admissible labeling of a polyomino is shown.

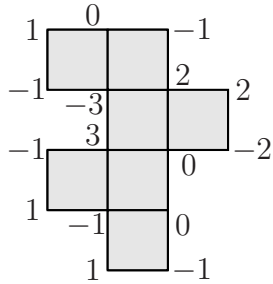
\begin{figure}[hbt]
\begin{center}
\psset{unit=0.8cm}
\begin{pspicture}(2,-1)(2,5)
\psline(1,0)(2,0)
\psline(0,1)(2,1)
\psline(0,2)(2,2)
\psline(0,4)(2,4)
\psline(0,1)(0,2)
\psline(0,3)(0,4)
\psline(1,0)(1,4)
\psline(2,0)(2,4)
\rput(0.75,0){$1$}
\rput(0.65,0.75){$-1$}
\rput(2.25,0.75){$0$}
\rput(2.4,0){$-1$}
\rput(-0.25,1){$1$}
\rput(-0.35,2){$-1$}
\rput(-0.35,3){$-1$}
\rput(-0.25,4){$1$}
\rput(0.75,2.25){$3$}
\rput(0.6,2.75){$-3$}
\rput(0.75,4.25){$0$}
\rput(2.35,1.75){$0$}
\rput(2.25,3.25){$2$}
\rput(2.4,4){$-1$}
\rput(3.4,2){$-2$}
\rput(3.25,3){$2$}
\pspolygon[style=fyp,fillcolor=light](1,0)(1,1)(2,1)(2,0)
\pspolygon[style=fyp,fillcolor=light](0,1)(0,2)(1,2)(1,1)
\pspolygon[style=fyp,fillcolor=light](1,1)(1,2)(2,2)(2,1)
\pspolygon[style=fyp,fillcolor=light](1,2)(1,3)(2,3)(2,2)
\pspolygon[style=fyp,fillcolor=light](2,2)(2,3)(3,3)(3,2)
\pspolygon[style=fyp,fillcolor=light](0,3)(0,4)(1,4)(1,3)
\pspolygon[style=fyp,fillcolor=light](1,3)(1,4)(2,4)(2,3)
\end{pspicture}
\end{center}
\caption{An admissible labeling}\label{admissible}
\end{figure}

An \textit{inner interval} $I$ of a polyomino $\MP$ is an interval with the property that all cells inside $I$ belong to $\MP$.

Let $I$ be an inner interval of a polyomino $\MP$. Then we introduce the admissible labeling $\alpha_{I}: V(\MP) \to \ZZ$ of $\MP$, which will be used in the proof of our main theorem, as follows:
\begin{equation}
\alpha_{I}(a)=\left \{\begin {array}{lll}
-1,&\text{if $a$ is a diagonal corner of $I$},\\
1,&\text{if $a$ is an anti-diagonal corner of $I$},\\
0,&\text{otherwise}.
\end{array}\right.
\nonumber
\end{equation}

\medskip
Now, we introduce a special labeling of a simple polyomino $\MP$, called a \textit{border labeling}. By Lemma~\ref{borderpolygon}, $B(\MP)$ is a rectilinear polygon. While walking counter clockwise around $B(\MP)$, we label the corners alternatively by $+1$ and $-1$ and label all the other vertices of $\MP$ by $0$. Since $B(\MP)$ has even number of vertices, this labeling is always possible for $\MP$. Also, it is obvious that every simple polyomino has exactly two border labelings. Figure~\ref{border} shows a border labeling of the polyomino which was displayed in Figure~\ref{admissible}.

\begin{figure}[hbt]
\begin{center}
\psset{unit=0.8cm}
\begin{pspicture}(2,-1)(2,5)
\psline(1,0)(2,0)
\psline(0,1)(2,1)
\psline(0,2)(2,2)
\psline(0,4)(2,4)
\psline(0,1)(0,2)
\psline(0,3)(0,4)
\psline(1,0)(1,4)
\psline(2,0)(2,4)
\rput(0.75,0){$1$}
\rput(0.65,0.75){$-1$}
\rput(2.25,0.75){$0$}
\rput(2.4,0){$-1$}
\rput(-0.25,1){$1$}
\rput(-0.35,2){$-1$}
\rput(-0.25,3){$1$}
\rput(-0.35,4){$-1$}
\rput(0.75,2.25){$1$}
\rput(0.6,2.75){$-1$}
\rput(0.75,4.25){$0$}
\rput(2.35,1.75){$1$}
\rput(2.35,3.25){$-1$}
\rput(2.3,4){$1$}
\rput(3.4,2){$-1$}
\rput(3.25,3){$1$}
\pspolygon[style=fyp,fillcolor=light](1,0)(1,1)(2,1)(2,0)
\pspolygon[style=fyp,fillcolor=light](0,1)(0,2)(1,2)(1,1)
\pspolygon[style=fyp,fillcolor=light](1,1)(1,2)(2,2)(2,1)
\pspolygon[style=fyp,fillcolor=light](1,2)(1,3)(2,3)(2,2)
\pspolygon[style=fyp,fillcolor=light](2,2)(2,3)(3,3)(3,2)
\pspolygon[style=fyp,fillcolor=light](0,3)(0,4)(1,4)(1,3)
\pspolygon[style=fyp,fillcolor=light](1,3)(1,4)(2,4)(2,3)
\end{pspicture}
\end{center}
\caption{A border labeling}\label{border}
\end{figure}
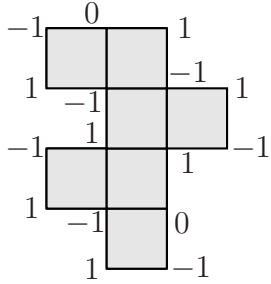

\begin{Lemma}
\label{borderlabeling}
A border labeling of a simple  polyomino is admissible.
\end{Lemma}

\begin{proof}
Let $\MP$ be a simple polyomino, and let $\alpha$ be a border labeling of $\MP$. Let $I$ be a maximal horizontal edge interval of $\MP$. We show that $\sum_{a\in I}\alpha(a)=0$. Let $I_1,\ldots,I_t$ be all maximal horizontal border edge intervals of $\MP$ which are contained in $I$. Note that the intervals $I_j$ are pairwise disjoint. Then $\sum_{a\in I}\alpha(a)=\sum_{a\in I_i \atop 1\leq i\leq t}\alpha(a)$, since the only elements of $I$ for which $\alpha(a)\neq 0$ are the corners of the rectilinear polygon $B(\MP)$, and since the endpoints of $I_1,\ldots,I_t$ are corners of $B(\MP)$. But, $\sum_{a\in I_i \atop 1\leq i\leq t}\alpha(a)=0$, since by definition of a border labeling, we have $\sum_{a\in I_i}\alpha(a)=0$, for each $i=1,\ldots,t$. Similarly, for a maximal vertical edge interval $I$ of $\MP$, we have $\sum_{a\in I}\alpha(a)=0$. Hence $\alpha$ is admissible.
\end{proof}

Now, we present the algebraic concepts and facts which are the main subject of this paper.

Let $\MP$ be a polyomino and $S=K[x_a:a\in V(\MP)]$ be the polynomial ring with the indeterminates $x_a$ over the field $K$. The  $2$-minor $x_{a}x_{b}-x_{c}x_{d}\in S$ is called an \textit{inner minor} of $\MP$ if $[a,b]$ is an inner interval of $\MP$ with anti-diagonal corners $c$ and $d$. Associated to $\MP$ is the binomial ideal $I_{\MP}$ in $S$, generated by all inner minors of $\MP$. This ideal is called the \textit{polyomino ideal} of $\MP$, and the $K$-algebra $K[\MP]=S/I_{\MP}$ is called the \textit{coordinate ring} of $\MP$.

In the sequel we use the following notation. Let $v\in \NN^m$ for some $m$. Then we set $\mathbf{x}^{v}=\prod_{i=1}^m{x_i}^{v_i}$ in the polynomial ring $K[x_1,\ldots,x_m]$. Note that a vector $v\in \ZZ^m$ can be written  uniquely as $v=v^+-v^-$ with $v^+,v^-\in \NN^m$ and such that the inner product of $v^+$ and $v^-$ is equal to zero.

Let $\alpha$ be an admissible labeling of a polyomino $\MP$. We may view $\alpha$ as a vector $\alpha\in \ZZ^n$, where $n$ is the number of vertices of $\MP$. By using this notation, we associate to $\alpha$ the binomial $f_{\alpha}={\mathbf{x}}^{\mathbf{\alpha}^{+}}-{\mathbf{x}}^{\mathbf{\alpha}^{-}}$ (see \cite{HQSh}). Let $J_\MP$ be the ideal in $S$ which is generated by the binomials $f_\alpha$, where $\alpha$ is an admissible labeling of $\MP$. It is known by \cite[Proposition~1.2]{HQSh} that $J_{\MP}$ is  the lattice ideal of a certain saturated lattice, and hence by \cite[Theorem~7.4]{MS}, $J_{\MP}$ is a prime ideal. By definition, it is clear that $I_\MP\subset J_\MP$. Following \cite{HQSh}, a polyomino $\MP$ is called \textit{balanced} if $f_{\alpha} \in I_{\MP}$ for every admissible labeling $\alpha$ of $\MP$.

To better understand the significan
ce of the notion balanced, we recall some concepts from \cite{EH}. Let $\MB\subset \ZZ^m$ for some $m$. Let $G_{\MB}$ be the graph with the vertex set $\mathbb{N}^m$ such that two vertices $\mathbf{a}$ and $\mathbf{c}$ are adjacent in $G_{\MB}$ if $\mathbf{a}-\mathbf{c}\in \pm \MB$. The vectors $\mathbf{a}$ and $\mathbf{c}$ are said to be \textit{connected via $\MB$} if they belong to the same connected component of $G_{\MB}$. The binomial ideal $I(\MB)$ in the polynomial ring $K[x_1,\ldots,x_m]$ is defined to be the ideal
\[
I(\MB)=(\mathbf{x}^{\mathbf{b}^+}-\mathbf{x}^{\mathbf{b}^-}: \mathbf{b}\in \MB).
\]

Now, let $\MP$ be a polyomino contained in the rectangular polyomino $\MI$ with $V(\MI)=[(1,1),(m,n)]$ for some positive integers $m$ and $n$. Let $I$ be an inner interval of $\MP$, and set $\mathbf{u}_{I}={(u_{I}^{(i,j)})}_{1\leq i\leq m \atop 1\leq j\leq n}\in \ZZ^{m\times n}$ where
\begin{equation}
u_{I}^{(i,j)}=\left \{\begin {array}{lll}
-1,&\text{if}~~~
(i,j)~\text{is a diagonal corner of}~I,\\
1,&\text{if}~~~(i,j)~\text{is an anti-diagonal corner of}~I,\\
0,&\text{otherwise}.
\end{array}\right.
\nonumber
\end{equation}
Note that if $I$ is just a cell $C$ of $\MP$, then with the notation of \cite{HQSh}, $\mathbf{u}_{I}=b_C$. It is known that the elements $b_C$ with $C\in \MI$ are linearly independent over $\ZZ$ (see \cite[Lemma~1.1]{HQSh}).

We set
\[
\MM(\MP)=\{\mathbf{u}: \mathbf{u}=\pm \mathbf{u}_{I}~\text{for some inner interval}~I~\text{of}~\MP\}.
\]

We need the following proposition to prove the main result of this paper.

\begin{Proposition}
\label{balanced}
Let $\MP$ be a polyomino. Then the following conditions are equivalent:
\begin{enumerate}
\item[{\em (a)}] $\MP$ is balanced;
\item[{\em (b)}] $I_\MP=J_\MP$;
\item[{\em (c)}] For each admissible labeling $\alpha$ of $\MP$, $\alpha^+$ and $\alpha^-$ are connected via $\MM(\MP)$;
\item[{\em (d)}] For each admissible labeling $\alpha$ of $\MP$, there exist $\mathbf{u}_1,\ldots,\mathbf{u}_t\in \MM(\MP)$ such that $\alpha^-+\mathbf{u}_1+\cdots+\mathbf{u}_i\in \mathbb{N}^n$ for all $i=1,\ldots,t$, and $\alpha^+=\alpha^-+\mathbf{u}_1+\cdots+\mathbf{u}_t$.
\end{enumerate}
\end{Proposition}

\begin{proof}
The conditions (a) and (b) are obviously equivalent. Also, (c) and (d) are equivalent by definition of $G_{\MM(\MP)}$ (see also the proof of \cite[Theorem~6.53]{EH}). We show that (a) and (c) are equivalent. Let $\alpha$ be an admissible labeling of $\MP$. By \cite[Theorem~6.53]{EH},
$f_{\alpha}=\mathbf{x}^{\alpha^+}-\mathbf{x}^{\alpha^-}\in I(\MM(\MP))$ if and only if $\alpha^+$ and $\alpha^-$ are connected via $\MM(\MP)$. But note that $I_{\MP}=I(\MM(\MP))$. So, $f_{\alpha}\in I_{\MP}$ if and  only if $\alpha^+$ and $\alpha^-$ are connected via $\MM(\MP)$. Hence, we have $\MP$ is balanced  if and  only if $\alpha^+$ and $\alpha^-$ are connected via $\MM(\MP)$ for every admissible labeling $\alpha$ of $\MP$.
\end{proof}

\section{Simple polyominoes}
\label{main section}

The following theorem which was conjectured in \cite{HQSh} is the main theorem of this paper.

\begin{Theorem}
\label{main}
A polyomino is simple if and only if it is balanced.
\end{Theorem}

\begin{proof}
Let $\MP$ be a  polyomino. First suppose $\MP$ is simple. We have to show that for any admissible labeling $\alpha$ of $\MP$ we have that $f_\alpha\in I_{\MP}$, and we show this by induction on $\deg f_{\alpha}$. Suppose $\deg f_\alpha=2$. Then $\alpha=\pm \alpha_{I}$ for some inner interval $I$, because $\MP$ is simple. Thus by definition $f_\alpha\in  I_{\MP}$.

Now suppose that $\deg f_\alpha>2$. We choose $a_0\in V(\MP)$ with $\alpha(a_0)>0$. Since $\alpha$ is admissible there exists a horizontal edge interval $[a_0,a_1]$ of $\MP$ with $\alpha(a_1)<0$. By using again that $\alpha$ is admissible, there exists a vertical edge interval $[a_1,a_2]$ of $\MP$ with $\alpha(a_2)>0$. Proceeding in this way we obtain a sequence of edge intervals of $\MP$,
\[
[a_0,a_1],[a_1,a_2], [a_2,a_3],\ldots
\]
which are alternatively  horizontal and vertical and such that $\sign (\alpha(a_i))=(-1)^{i}$ for all $i$.

Since $V(\MP)$ is a finite set, there exists a smallest integer $r$ such that $[a_r,a_{r+1}]$ intersects $[a_j,a_{j+1}]$ for some $j<r-1$. We may assume that $j=0$. If $[a_r,a_{r+1}]$ is a vertical interval, then  $[a_r,a_{r+1}]$  and $[a_0,a_1]$ intersect in precisely one vertex, which we call $a$. If $[a_r,a_{r+1}]$ is horizontal, then we let $a=a_1$. In this way we obtain a simple rectilinear polygon $R$ whose edges are edge intervals of $\MP$ with corner sequence $a,a_1,a_2,\ldots,a_{r-1},a$ if $[a_r,a_{r+1}]$ is vertical and corner sequence  $a,a_2,a_3,\ldots,a_{r-1},a$ if $[a_r,a_{r+1}]$ is horizontal. Moreover, we have  $\sign (\alpha(a_i))=(-1)^{i}$ for all $i$. The  cells in the interior of $R$ all belong to $\MP$ because $\MP$ is simple. We may assume that the orientation of  $R$ given by the order of the corner sequence is counterclockwise. Then with respect to this orientation the interior of $R$ meets $R$ on the left hand side,  see Figure~\ref{rectilinear}.

We call a convex corner $c$ of $R$ {\em good} if the rectangle which is spanned by $c$ and its neighbor corners is in the interior of $R$. We claim that $R$ has at least four good  corners. We will prove the claim later and first discuss its consequences. Since $R$ has at least four  good corners there is at least one good corner $c$ such that $c$ and its neighbor corners are all different from $a$. Let $I$ be the rectangle in the interior of $R$ spanned by $c$ and its neighbor corners. Without loss of generality we may assume that this corner looks like the one displayed  in  Figure~\ref{good} with $c=a_i$.

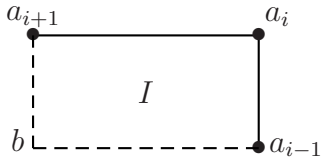
\begin{figure}[hbt]
\begin{center}
\psset{unit=0.5cm}
\begin{pspicture}(2.5,-1)(4.5,5)
\psline(6,1)(6,4)
\psline(6,4)(0,4)
\rput(7,1){$a_{i-1}$}
\rput(6.5,4.5){$a_{i}$}
\rput(0,4.5){$a_{i+1}$}
\rput(-0.4,1.2){$b$}
\rput(6,1){$\bullet$}
\rput(6,4){$\bullet$}
\rput(0,4){$\bullet$}
\rput(3,2.5){$I$}
\psline[linestyle=dashed](0,1)(0,4)
\psline[linestyle=dashed](0,1)(6,1)
\end{pspicture}
\end{center}
\caption{A good corner and its rectangle}\label{good}
\end{figure}

Since all cells in the interior of  $I$ belong to the interior of  $R$ and since all those cells belong to $\MP$, it follows that $f_{\alpha_{I}}\in I_{\MP}$. Without loss of generality, we may assume that $\alpha(a_i)<0$, and hence $\alpha(a_{i-1}), \alpha(a_{i+1})>0$. Then the homogeneous binomial $g=f_{\alpha}-(\mathbf{x}^{\alpha^+}/x_{a_{i-1}}x_{a_{i+1}})f_{\alpha_{I}}$ has the same degree as $f_\alpha$ and belongs to $J_\MP$, since  $f_{\alpha}$ and $f_{\alpha_{I}}$ belong to $J_{\MP}$. Furthermore, $g=x_{a_{i}}h$, where $h=x_b(\mathbf{x}^{\alpha^+}/x_{a_{i-1}}x_{a_{i+1}})-\mathbf{x}^{\alpha^-}/x_{a_{i}}$. It follows that  $h\in J_{\MP}$, since $x_{a_i}\notin J_{\MP}$ and since $J_{\MP}$ is  a prime ideal. Since $J_\MP$ is generated by the binomials $f_\beta$ with $\beta$ an admissible labeling of $\MP$,  there exist $f_{\beta_l}\in J_{\MP}$ such that $h=\sum_{l=1}^{s}r_lf_{\beta_l}$, where $\deg f_{\beta_l}\leq \deg h$ and $r_l\in S$ for all $l$. Since $\deg h< \deg f_{\alpha}$ we also have  $\deg f_{\beta_l}< \deg f_{\alpha}$ for all $l$. Thus our induction hypothesis  implies that  $f_{\beta_l}\in I_{\MP}$ for all $l$. It follows that $h\in I_{\MP}$, and hence $f_{\alpha}\in I_{\MP}$, since $f_{\alpha_{I}}\in I_{\MP}$.

In order to complete the proof that $\MP$ is balanced it remains to prove that indeed any rectilinear polygon $R$ has at least four good convex corners. We prove this by defining an injective  map $\gamma$ which assigns to each convex corner  of $R$ which is not good a concave corner of $R$. Since, as is well known and easily seen,  for any simple rectilinear polygon the number of convex corners is four more than the number of concave corners, it  will follow that there are at least four good corners.

The map $\gamma$ is defined as follows: let $c$ be a convex corner of $R$ which is not good. Then the polygon $R$ crosses the open interior of the rectangle which is spanned by $c$ and the neighbor corners of $c$. The gray area in Figure~\ref{cross} belongs to the interior of $R$.

\begin{figure}[hbt]
\begin{center}
\psset{unit=0.5cm}
\begin{pspicture}(2.65,-1)(4.5,5)
\psline(6,1)(6,4)
\psline(6,4)(0,4)
\rput(6,1){$\bullet$}
\rput(6,4){$\bullet$}
\rput(6.4, 4.4){$c$}
\rput(0,4){$\bullet$}
\psline[linestyle=dashed](0,1)(0,4)
\psline[linestyle=dashed](0,1)(6,1)
\psline(-1,3)(3,3)
\psline(3,3)(3,2)
\psline(3,2)(5,2)
\psline(5,2)(5,0)
\pspolygon[style=fyp,fillcolor=light](0,4)(6,4)(6,1)(5,1)(5,2)(3,2)(3,3)(0,3)
\end{pspicture}
\end{center}
\caption{$R$ intersects the rectangle}\label{cross}

\end{figure}
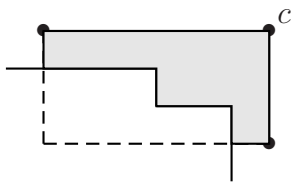

Now we let $L$ be the angle bisector of the $90$ degrees angle centered in $c$. Next we consider the set $\ML_c$ of all lines perpendicular to $L$. The unique  line in $\ML_c$ which intersects $L$ in the point $p$ and such that the distance from $c$ to $p$ is $t$, will be denoted by $L_t$. There is a smallest number $t_0$ such that $L_{t_0}$ has a non-trivial intersection with $R$ in the open interior of the rectangle. This intersection with $L_{t_0}$ consists of at least one and at most finitely many concave corners of $R$, see Figure~\ref{half}.

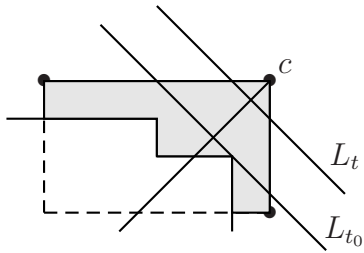
\begin{figure}[hbt]
\begin{center}
\psset{unit=0.5cm}
\begin{pspicture}(2.5,-1)(4.5,5)
\psline(6,0.5)(6,4)
\psline(6,4)(0,4)
\rput(6,0.5){$\bullet$}
\rput(6,4){$\bullet$}
\rput(6.4, 4.4){$c$}
\rput(0,4){$\bullet$}
\psline[linestyle=dashed](0,0.5)(0,4)
\psline[linestyle=dashed](0,0.5)(6,0.5)
\psline(-1,3)(3,3)
\psline(3,3)(3,2)
\psline(3,2)(5,2)
\psline(5,2)(5,0)
\pspolygon[style=fyp,fillcolor=light](0,4)(6,4)(6,0.5)(5,0.5)(5,2)(3,2)(3,3)(0,3)
\psline(6,4)(2,0)
\psline(3,6)(8,1)
\rput(8,2){$L_t$}
\psline(1.5,5.5)(7.5,-0.5)
\rput(8,0){$L_{t_0}$}
\end{pspicture}
\end{center}
\caption{$L_{t_0}$ defines $\gamma(c)$}\label{half}
\end{figure}

We define $\gamma$ to assign to $c$ one of these concave corners. The map $\gamma$ is injective. Indeed, if $d$ is another convex corner of $R$ with $\gamma(d)=\gamma(c)$, then the line in $\ML_d$ which hits $\gamma(c)$ must be identical with $L_{t_0}$, and this implies that $d$ lies in the intersection of the rectangle with the linear half space defined by $L_{t_0}$ containing $c$. But in this area there is no other corner of $R$ which is not good. Hence $d=c$.

\medskip
Conversely, suppose now that $\MP$ is balanced and assume that $\MP$ is not simple. Let $\MP'$ be a hole of $\MP$. Then by Lemma~\ref{hole}, $\MP'$ is a simple polyomino. Let $\alpha$ be a border labeling of $\MP'$. We consider the labeling $\beta$ of $\MP$ which for each $a\in V(\MP)$ is defined as follows:
\begin{equation}
\beta(a)=\left \{\begin {array}{lll}
\alpha(a)&\text{if}~~~
a\in V(\MP'),\\
0&\text{if}~~~
a\notin V(\MP').
\end{array}\right.
\nonumber
\end{equation}
Then $\beta$ is an admissible labeling of $\MP$, by a similar argument as in the proof of Lemma~\ref{borderlabeling}. Indeed, let $I$ be a maximal horizontal (vertical) edge interval of $\MP$ and let $\MS$ be the set of all horizontal (vertical) border edge intervals of $\MP'$ such that $I_j\cap I\neq \emptyset$. If $\MS=\emptyset$, then $\beta(a)=0$ for all $a\in I$. If $\MS\neq \emptyset$ and $I_j\in \MS$, then $I_j\subset I$. Since the intervals $I_j$ are disjoint, we have $\sum_{a\in I}\beta(a)=\sum_{a\in I_j \atop I_j\in \MS}\beta(a)=\sum_{a\in I_j \atop I_j\in \MS}\alpha(a)$. Hence $\sum_{a\in I}\beta(a)=0$, because by definition of $\alpha$, we have $\sum_{a\in I_j}\alpha(a)=0$ for all $I_j\in \MS$.

Note that we may consider $\alpha$ and $\beta$ as vectors in $\ZZ^{m\times n}$ where $m$ and $n$ are positive integers with $V(\MP)\subset [(1,1),(m,n)]$. Since $\MP$ is a balanced polyomino, it follows that there exist $\mathbf{u}_1,\ldots,\mathbf{u}_t\in \MM(\MP)$ such that $\beta^+=\beta^-+\mathbf{u}_1+\cdots+\mathbf{u}_t$, by Proposition~\ref{balanced}. On the other hand, since $\MP'$ is a simple polyomino, it follows from the first part of the proof that $\MP'$ is also balanced. Thus by Proposition~\ref{balanced} there exist $\mathbf{u}'_1,\ldots,\mathbf{u}'_l\in \MM(\MP')$ such that $\alpha^+=\alpha^-+\mathbf{u}'_1+\cdots+\mathbf{u}'_l$, since $\alpha$ is admissible by Lemma~\ref{borderlabeling}. Note that by the construction of the labeling $\beta$, it is clear that $\beta^+=\alpha^+$ and $\beta^-=\alpha^-$ as vectors in $\ZZ^{m\times n}$. So we have $\mathbf{u}_1+\cdots+\mathbf{u}_t=\mathbf{u}'_1+\cdots+\mathbf{u}'_l$. For each $i=1,\ldots,t$, we have $\mathbf{u}_i=\pm \mathbf{u}_{{I}_i}$, and for each $j=1,\ldots,l$, we have $\mathbf{u}'_j=\pm \mathbf{u}_{{I}'_j}$, where ${I}_i$ and ${I}'_j$ are inner intervals of $\MP$ and $\MP'$, respectively. So, it follows that for each $i,j$, $\mathbf{u}_i$ and $\mathbf{u}'_j$ are linear combination of the $b_{C}$'s and $b_{C'}$'s, respectively, where $C$ stands for cells of $\MP$ and $C'$ stands for cells of $\MP'$. But the $b_C$'s and $b_{C'}$'s are linearly independent, so that $\mathbf{u}_1+\cdots+\mathbf{u}_t=\mathbf{u}'_1+\cdots+\mathbf{u}'_l=0$, which is a contradiction, since obviously we have $\beta^+\neq \beta^-$. Therefore, $\MP$ is a simple polyomino.
\end{proof}

By the above theorem together with \cite[Corollary~2.3]{HQSh}, we get the following.

\begin{Corollary}
\label{Ayesha's conj}
Let $\MP$ be a simple polyomino. Then $K[\MP]$ is a Cohen--Macaulay normal domain.
\end{Corollary}

\end{document}